\documentclass[final]{siamltex}

\usepackage{a4wide}
\usepackage{graphicx}
\usepackage{amstext, amsmath, amssymb}
\usepackage[shadow]{todonotes}
\usepackage{booktabs}
\usepackage{url}
\usepackage{import}
\usepackage{amsfonts}
\usepackage{amsbsy}
\usepackage{fancyvrb}
\usepackage{stmaryrd}
\usepackage{pdfsync}
\usepackage{graphicx}

\usepackage[shadow]{todonotes}
\usepackage[numbers,square,sort&compress]{natbib}

\numberwithin{equation}{section}
\numberwithin{figure}{section}
\numberwithin{table}{section}

\numberwithin{theorem}{section}

\newcommand{\mcT}{\mathcal{T}}
\newcommand{\mcF}{\mathcal{F}}

\newcommand{\R}{\mathbb{R}}

\newcommand{\foralls}{\forall\,}
\newcommand{\dx}{\,\mathrm{d}x}
\newcommand{\ds}{\,\mathrm{d}s}

\DefineVerbatimEnvironment{code}{Verbatim}{frame=single,rulecolor=\color{blue}}

\newcommand{\mesh}{\mathcal{T}}

\newcommand{\pO}{\partial \Omega}

\newcommand{\bfI}{\boldsymbol{I}}

\newcommand{\tn}{|\mspace{-1mu}|\mspace{-1mu}|} % Triple norm

\newcommand{\deltahone}{{{\delta}_h'}}

\begin{document}
\title{$\bf L^2$-Error Estimates for Finite Element Approximations
of Boundary Fluxes}

\author{
  Mats G.\ Larson\thanks{Department of Mathematics, Ume{\aa}
  University, SE-901 97 Ume{\aa}, Sweden. email: mats.larson@math.umu.se}
  \and
  Andr\'e Massing\thanks{Simula Research Laboratory, P.O. Box 134, 1325
  Lysaker, Norway. email: massing@simula.no}
}

%\date{Received: \today / Accepted: }

\maketitle

\begin{abstract}
\noindent
We prove quasi-optimal a priori error estimates for finite element
approximations of boundary normal fluxes in the $L^2$-norm. Our
results are valid for a variety of different schemes for weakly
enforcing Dirichlet boundary conditions including Nitsche's method,
and Lagrange multiplier methods. 
The proof is based on an error representation formula that is derived
by using a discrete dual problem with $L^2$-Dirichlet boundary data
and combines a weighted discrete stability estimate for the dual problem
with anisotropic interpolation estimates in the boundary zone.
\end{abstract}



%------------------------------------------------------------------------------
\section{Introduction}
The normal flux at the boundary or on interior
interfaces is in general of great interest in applications. Examples
include surface stresses in mechanics, heat transfer through interfaces,
and transport of fluid in Darcy flow.

Recently Melenk and Wohlmuth \cite{MelenkWohlmuth2012} has shown quasi-optimal
order estimates for fluxes in a mortar setting where continuity and
boundary conditions is enforced using a mortaring space of Lagrange
multipliers. More precisely, they shown that the $L^2$-norm of the
error in the normal flux is of order $|\ln h| h$ for piecewise linear
polynomials and of order $h^k$ for piecewise polynomials of order $k$.
In contrast only $h^{k-1/2}$ will be obtained if a trace
inequality is used in combination with standard convergence theory
for saddle point problems, see \cite{Bramble1981}.

In this contribution we give an alternative proof of this result and,
focusing on the case $k=1$, we also consider a wider variety of
methods for weakly enforcing Dirchlet boundary conditions,
including Nitsche's method and stable and stabilized Lagrange
multipliers methods.

Our proof is based on an error representation formula where the error
in the normal flux is represented in terms of the interpolation error
and the solution to a discrete dual problem with $L^2$-Dirichlet
boundary data. Key to the error estimate is a stability estimate for
the discrete dual problem in terms of the $L^2$-norm of the
Dirichlet data. In the continuous case such an estimate is known,
see Chabrowsky \cite{Chabrowski1982} and \cite{Chabrowski1991}, and provides control
of the gradient weighted with the distance to the boundary as well
as a max-norm control of the $L^2$-norms of the solution on manifolds
close to and parallel with the boundary. 
We prove a corresponding stability estimate for our discrete dual
problem. In contrast to the approach by \citet{MelenkWohlmuth2012},
we avoid using a Besov space framework.

Our error representation formula is related to the one derived in the
\citet{GilesLarsonLevenstamEtAl1997,CareyChowSeager1985,
PehlivanovLazarovCareyEtAl1992} where various estimates
for functionals of the normal flux are derived and
\cite{EstepTavenerWildey2010} where adaptive methods based on dual
problems targeting the flux in a coupled problem are developed.
Note however that in our setting where we seek an a priori estimate,
we employ a discrete dual problem while in the a posteriori setting,
the corresponding continuous dual problem is used. Here we also
establish the stability of the discrete dual problem using analytical
techniques while in the duality based a posteriori error estimates,
stability is often estimated using computational techniques or a known
analytical stability result.

The remainder of this work is organized as follows.  In
Section~\ref{sec:model-problem} we introduce the model problem and its
variational formulations we will consider throughout this work.
Corresponding finite element discretizations are presented in
Section~\ref{sec:fem-discretization} together with the definition of
the discrete boundary fluxes.  In Section~\ref{sec:stability-bounds}
we prove stability bounds for the discrete dual problem and provide
interpolation estimates of the solution close to the boundary.
Combining these results allows us to prove $L^2$-error estimates for
the boundary flux approximations in Section~\ref{sec:error-estimates}.
In Section~\ref{sec:numerical-results}, we finally present numerical
results illustrating the theoretical findings.

\section{Model Problem}
\label{sec:model-problem}
Let $\Omega$ be a polygonal domain in ${\R}^d$, $d = 2,3$ with boundary
$\partial \Omega$.
We consider the elliptic model problem: find $u: \Omega \rightarrow {\R}$
such that
\begin{align}
-\Delta u &= f \quad \text{in $\Omega$}
\label{eq:model-problem-strong-pde}
\\
u&=g\quad \text{on $\partial \Omega$}
\label{eq:model-problem-strong-bc}
\end{align}
where $f$ and $g$ are given data.  Then the
boundary flux $\sigma_n$ for the solution $u$ is defined by
\begin{equation}
  \sigma_n = n \cdot \nabla u
  \label{eq:sigman}
\end{equation}
where $n$ is the outwards pointing unit normal to $\partial \Omega$.

In what follows, we consider the standard Sobolev spaces $H^s(U)$, $s
\geqslant 0$ on some domain $U$, endowed the the usual norms $\| \cdot
\|_{s,U}$ and semi-norms $ | \cdot |_{s,U}$.
More generally, the space
$W^{s,p}(U)$ is defined as the Sobolev space consisting of all
functions having $p$-integrable derivates up to order $s$ on $U$.
As usual, $H^0(\Omega) =
L^2(\Omega)$ and $H^{-1/2}(\partial \Omega)$ denotes the dual space of
$H^{1/2}(\partial \Omega)$.
Moreover, for a function $g \in H^{1/2}(\partial \Omega)$
we introduce the notation
$H^1_g(\Omega) = \{ v \in H^1(\Omega) : v |_{\partial \Omega} = g\}$.
The scalar product in $H^s(U)$ is written as $(\cdot,
\cdot)_{s,U}$ and
to simplify the notation, we generally omit the domain designation
if $U = \Omega$ and the Sobolev index if $s = 0$ in both norm and
scalar product expressions.
Using this notation, a weak formulation of the elliptic boundary value
problem~\eqref{eq:model-problem-strong-pde}--\eqref{eq:model-problem-strong-bc} is to seek
$u \in H^1_g(\Omega)$
 such that
\begin{align}
  a(u,v) = l(v) \quad \foralls v \in H^1_0(\Omega)
  \label{eq:model-problem-weak-form}
\end{align}
where
\begin{align}
  a(u,v) &= (\nabla u, \nabla v)
  \\
  l(v) &= (f,v)
\end{align}
Here, the boundary condition $u|_{\partial \Omega} = g$ is already
incorporated into the trial space $H^1_g(\Omega)$.
Alternatively, the boundary
condition~\eqref{eq:model-problem-strong-bc} can be enforced weakly by
using a Lagrange multiplier approach. Introducing the bilinear form
\begin{equation}
  b(\mu, v) = (\mu,v)_{\partial \Omega}
\end{equation}
the resulting variational
formulation is given by the saddle point problem: find
$(u,\lambda) \in H^1(\Omega) \times H^{-1/2}(\partial \Omega)$ such that
\begin{align}
  a(u,v) + b(\lambda,v) + b(\mu, u) = l(v)  +  b(\mu, g)
  \quad \foralls (v,\mu) \in H^1(\Omega) \times H^{-1/2}(\partial
  \Omega)
  \label{eq:model-problem-lm-form}
\end{align}
For brevity, we might denote the left-hand side 
by $A(u,\lambda; v, \mu)$ and the right-hand side $L(v,\mu)$.
It is well-known \cite{Babuska1973,Brezzi1974,Pitkaeranta1979,Stenberg1995}, that
the saddle point problem~\eqref{eq:model-problem-lm-form} satisfies
the  Babu\v{s}ka-Brezzi condition, in particular
\begin{equation}
  \sup_{v \in H^1(\Omega) \setminus \{0\}} \dfrac{b(\lambda,v)}{\| v \|_{1,\Omega}} \gtrsim
  \| \mu \|_{-1/2,\partial \Omega}
  \quad \foralls \mu \in H^{-1/2}(\partial \Omega)
  \label{eq:inf-sup-cont}
\end{equation}
Consequently, problem~\eqref{eq:model-problem-lm-form} possesses a
unique solution $(u,\lambda)$, where $u$
solves~\eqref{eq:model-problem-strong-pde}--\eqref{eq:model-problem-strong-bc}
in a weak sense and the Lagrange multiplier $\lambda$
represents the negative of the normal flux of $u$, i.e. $\lambda = -
\sigma_n$.

\section{Finite Element Discretizations of the Model Problem}
\label{sec:fem-discretization}
In this section, we introduce the finite element
discretizations of
problem~\eqref{eq:model-problem-strong-pde}--\eqref{eq:model-problem-strong-bc}
we will consider throughout this work.
The discretizations
are defined on a quasi-uniform partition $\mcT$ of $\Omega$ into shape
regular triangles in two or tetrahedra in three space dimensions with
mesh parameter $h$. For a given mesh $\mcT$, let the associated finite
element space of piecewise linear continuous functions be denoted by
$V_h$. We do not assume $V_h \subset H^1_g(\Omega)$
and consequently, the discretizations to be considered
will enforce the boundary condition~\eqref{eq:model-problem-strong-bc}
weakly. For each discretization we will define a discrete counterpart
$\Sigma_n$ of the boundary flux~\eqref{eq:sigman}.

\subsection{Nitsche's Method}
The \citet{Nitsche1971}  finite element method takes the form: find $u_h \in V_h$ such
that
\begin{equation}
  a_h(u_h,v) = l_h(v) \quad \forall v \in V_h
  \label{eq:nitsche-form}
\end{equation}
where the forms are defined by
\begin{align}\label{eq:ah}
  a_h(u,v) &= a(u, v) - (n \cdot \nabla u,v)_{\partial \Omega}
  - (n\cdot \nabla u,v)_{\partial \Omega} + \beta h^{-1}
  (u,v)_{\partial \Omega}
  \\
  \label{eq:lh}
  l_h(v) &= l(v) - (g,n \cdot \nabla v)_{\partial \Omega} + \beta
  h^{-1} (g,v)_{\partial\Omega}
\end{align}
with $\beta$ being a positive parameter. Introducing the energy norm
\begin{equation}
  \tn v \tn^2
  = \| \nabla v \|_\Omega^2 + h \| n \cdot \nabla v \|_{\partial
  \Omega}^2 + h^{-1} \| v \|_{\partial \Omega}^2
\end{equation}\label{eq:coercivity}
we recall that the bilinear form $a_h(\cdot,\cdot)$ is continuous
\begin{align}
  a_h(u,v) \lesssim \tn u \tn \; \tn v \tn
\end{align}
and that if the stabilization parameter $\beta$ is large enough,
a coercivity condition
\begin{align}
  \tn v \tn^2 \lesssim  a_h(v,v) \quad \forall v \in V_h
  \label{eq:nitsche-coerciv}
\end{align}
is satisfied, yielding the standard error estimate
\begin{equation}
  \tn u - u_h \tn \lesssim h  \| u \|_{2}
  \label{eq:nitsche-a-priori-est}
\end{equation}
Here and throughout, we use the notation $ a \lesssim b$ for $ a
\leqslant C b$ for some generic constant~$C$ which vary with the
context but is always independent of the mesh size $h$.
For proofs of~\eqref{eq:nitsche-coerciv}
and~\eqref{eq:nitsche-a-priori-est}, we refer to \cite{Nitsche1971,
Hansbo2005}.
To Nitsche's method~\eqref{eq:nitsche-form}, we associate the discrete
variational normal flux of the form
\begin{equation}\label{eq:Sigman}
  (\Sigma_n, v)_{\partial \Omega} = (\nabla u_h, \nabla v )_{\Omega}
  - (u_h - g,n\cdot \nabla v)_{\partial \Omega} - (f,v)_{\Omega}
  \quad \forall v \in V_h
\end{equation}
where $\Sigma_n$ is the so-called Nitsche flux
\begin{equation}
  \label{eq:flux-definition-nitsche-case}
  \Sigma_n = n \cdot \nabla u_h - \beta h^{-1} ( u_h - g )
\end{equation}

\subsection{Lagrange Multiplier Method}
To formulate a finite element discretization of the saddle point
problem~\eqref{eq:model-problem-lm-form}, we assume that
a discrete function space $\Lambda_h \subset L^2(\partial \Omega) \cap
H^{-1/2}(\partial \Omega)$ is given, and we equip $V_h$, $\Lambda_h$ and
the total approximation
space $V_h \times \Lambda_h$ with the natural norms
\begin{align}
  \tn u \tn^2 &= \| \nabla u \|^2_\Omega + \| h^{-1/2} u
  \|^2_{\partial \Omega}
  \\
  \tn \lambda \tn ^2 &= \| h^{1/2} \lambda \|^2_{\partial \Omega}
  \\
  \tn (u,\lambda) \tn^2 &= \tn u \tn^2 + \tn \lambda \tn^2
\end{align}
respectively, see \citet{Pitkaeranta1979}.
Employing the discrete norms
$\tn v \tn$ and $ \tn \mu \tn $, it is well-known
\cite{Pitkaeranta1979,Pitkaeranta1980} that the approximation space $\Lambda_h$
has to be designed carefully in order to satisfy the
discrete equivalent of the inf-sup condition~\eqref{eq:inf-sup-cont}.
Therefore, a stabilized Lagrange multiplier method has been proposed
by \citet{BarbosaHughes1991,BarbosaHughes1992} where residual terms
were added to circumvent the inf-sup condition~\eqref{eq:infsup}.
Recently, a generalized approach based on projection stabilized
Lagrange multipliers has been proposed by \citet{Burman2013}.

To cover a broad range of stable and stabilized Lagrange multiplier
methods, we assume that the discrete saddle point problem is of the
following form:
find $(u_h,\lambda_h) \in V_h \times {\Lambda}_h$
such that
\begin{equation}
  A_h(u_h,\lambda_h; v, \mu) = L_h(v,\mu)
  \quad \foralls (v,\mu) \in {V}_h \times
  {\Lambda}_h
  \label{eq:lm-weak-form}
\end{equation}
where
\begin{align}
  A_h(u,\lambda; v, \mu)
  &= a(u_h, v) + b(\lambda_h, v) + b(\mu, u_h) - c_h(u,\lambda;v,\mu)
  \label{eq:lm-weak-form-lhs}
  \\
  L_h(v,\mu)
  &=
  l(v) + (g,\mu)_{\partial \Omega}
  \label{eq:lm-weak-form-rhs}
\end{align}
Then, the approximation of the normal flux~\eqref{eq:sigman} is 
naturally defined by the negative of the
discrete Lagrange multiplier:
\begin{equation}
  \Sigma_n = - \lambda_h
  \label{eq:flux-definition-lm-case}
\end{equation}
In the variational form~\eqref{eq:lm-weak-form-lhs}, the bilinear form
$c_h(\cdot,\cdot)$ represents a consistent, possibly vanishing
stabilization form such the inf-sup condition 
\begin{equation}
  \label{eq:infsup}
  \sup_{(v,\mu)\in V_h\times \Lambda_h\setminus\{(0,0)\}}
  \frac{A_h(u,\lambda; v,\mu)}{\tn (v,\mu) \tn}
  \gtrsim \tn (u,\lambda) \tn
\end{equation}
holds, as well as the continuity condition
\begin{equation}
  A_h(u,\lambda; v,\mu) \lesssim 
  \tn (u,\lambda) \tn \, 
  \tn (v,\mu) \tn
  \label{eq:continuity-lm}
\end{equation}
and the 
error estimate
\begin{equation}
  \tn (u-u_h,\lambda-\lambda_h) \tn \lesssim h \| u \|_{2,\Omega} +
  h^{3/2} \| \lambda \|_{1,\partial \Omega}
  \label{eq:error-estimate-lm}
\end{equation}
Well-known Lagrange multiplier discretizations
which are covered by these assumptions are described and analyzed 
in \cite{Pitkaeranta1979,Pitkaeranta1980} and
\cite{BarbosaHughes1991,BarbosaHughes1992, Stenberg1995}.
In \cite{Pitkaeranta1979,Pitkaeranta1980}, 
Pitk\"aranta proved certain local stability conditions, roughly stating
that the pairing $P^1_c(\mcT_h) \times P^0_{\mathrm{dc}}(\Gamma_H)$ is
stable, if the mesh size $H$ of a given discretization $\Gamma_H$ of
the boundary $\partial \Omega = \Gamma$ satisfies the condition $h
\leqslant c H$ for some $c > 1$.
To avoid additional meshing of the boundary and to use 
the natural
space
\begin{equation*}
  \Lambda_h = \{ \mu \in L^2(\partial \Omega) | \; \mu \in P^{0}(F)
  \; \foralls F \in \partial \mcT_h\}
\end{equation*}
defined on the trace mesh $\partial \mesh$,
a stabilized symmetric Lagrange Multiplier
approach was proposed by \citet{BarbosaHughes1991,BarbosaHughes1992}.
\citet{Stenberg1995} simplified the approach by showing that the
weak formulation~\eqref{eq:lm-weak-form} combined with the stabilization form 
\begin{equation}
  c_h(u,\lambda;v,\mu) = \alpha h (\lambda + n \cdot \nabla u, \mu + n \cdot \nabla v)_{\partial \Omega}
  \label{eq:stabilization-form-stenberg}
\end{equation}
satisfies the inf-sup condition~\eqref{eq:infsup}, the continuity
condition~\eqref{eq:continuity-lm} and thus the error
estimate~\eqref{eq:error-estimate-lm} when $0 < \alpha < C_I$,
with $C_I$ being the constant
in~\eqref{eq:inverse-estimate-for-facets}.

Finally, we would like to mention the general approach
by \citet{Burman2013}. In this method,
the stabilization operator is given by some symmetric form
$c_h(\lambda, \mu)$
which, roughly speaking, controls the distance 
between a given discretization space $\Lambda_h$ and
another discrete space $L_h$ where $V_h \times L_h$ presents an
inf-sup stable pairing. Generally, the stabilization form is only
required to be optimal weakly consistent and
to not clutter the presentation, we skip
the details for the trivial adaption of our approach to this variant.

\section{Error Representation Formulas}
\label{sec:error-representation-formulas}
In this section, we establish the error representation formulas for
the discrete boundary fluxes.  The representation formula will later
allow us to bound the $L^2$-error of the boundary flux approximations
in terms of interpolation errors and a stability estimate for the
discrete solution to a suitable dual problem.

\subsection{Nitsche's Method}
For given boundary data $\psi \in L^2(\partial \Omega)$,
we define the discrete dual problem for Nitsche's method as follows:
find $\phi_h \in V_h$ such that
\begin{equation}
a_h(v, \phi_h) = m_{\psi,h}(v) \quad \forall v \in V_{h}
\label{eq:discrete-dual-problem-nitsche}
\end{equation}
where $a_h(\cdot,\cdot)$ is defined in (\ref{eq:ah}) and
\begin{equation}
m_{\psi,h}(v) = \beta h^{-1} (\psi, v )_{\partial \Omega}
- (\psi, n \cdot \nabla v )_{\partial \Omega}
\end{equation}

Setting $v = e_h = \pi_h u - u_h$ we obtain
\begin{equation}
a_h (e_h, \phi_h) = m_{\psi,h}(e_h)
\end{equation}
Using Galerkin orthogonality, we note that the left hand side
can be written
\begin{equation}
a_h(e_h, \phi_h) = a_h(\pi_h u - u, \phi_h)
\end{equation}
and for the right hand side
\begin{equation}
m_{\psi,h}(e_h) = m_{\psi,h}(\pi_h u - u) + m_{\psi,h}(u - u_h )
\end{equation}
where the second term takes the form
\begin{align}
m_{\psi,h}( u - u_h )
&= (\beta h^{-1} (u - u_h),\psi) - (n \cdot \nabla (u - u_h) , \psi )_{\partial \Omega}
\\
&= (\beta h^{-1} (g - u_h) + n \cdot \nabla u_h -   n \cdot \nabla u,\psi)_{\partial \Omega}
\\
&=(\Sigma_n(u_h) - \sigma_n(u), \psi)_{\partial \Omega}
\end{align}
Collecting these identities, we arrive at the error representation formula
\begin{equation}
(\sigma_n(u) - \Sigma_n(u_h), \psi)_{\partial \Omega}
= a_h(u - \pi_h u , \phi_h) - m_{\psi,h}(u - \pi_h u)
\label{eq:nitsche-flux-error-repres}
\end{equation}
Thus we have the following
\begin{lemma}
  \label{lemma:errorrep}
  With $\sigma_n(u)$ and $\Sigma_n(u_h)$
  defined by (\ref{eq:sigman}) and (\ref{eq:Sigman}) it holds
\begin{equation}
  \label{eq:errorrep}
\|\sigma_n(u) - \Sigma_n(u_h) \|_{\partial \Omega}
\leqslant \sup_{\psi \in L^2(\partial\Omega) \setminus \{0\}} \frac{1}{\|\psi\|_{\partial \Omega}}
\Big( |a_h(u - \pi_h u , \phi_h)| + |m_{\psi,h}(u - \pi_h u)| \Big)
\end{equation}
\end{lemma}

\subsection{Lagrange Multiplier Method}
We consider the following discrete dual problem: 
find $(\phi_h, \theta_h) \in {V}_h \times
\Lambda_h$ such that
\begin{align}
\label{eq:discrete-dual-problem-lagrange}
  A_h(v, \mu; \phi_h, \theta_h) = m_{\psi,h}(\mu)
  \quad \foralls (v,\mu) \in {V}_h \times
  {\Lambda}_h
\end{align}
where $A_h(\cdot,\cdot)$ is defined as in~\eqref{eq:lm-weak-form-lhs}
and
\begin{equation}
  m_{\psi,h}(\mu)
  =  (\psi, \mu)_{\partial \Omega}
  \label{eq:lm-discrete-dual-rhs}
\end{equation}
Setting $(v,\mu) = (\pi_h u - u_h, \pi_h \lambda - \lambda_h)$ and
using Galerkin orthogonality, we obtain
\begin{align*}
  m_{\psi,h}(\pi_h \lambda - \lambda_h)
  &= A_h(\pi_h u - u_h, \pi_h \lambda - \lambda_h; \phi_h, \theta_h)
  \\
  &= A_h(\pi_h u - u, \pi_h \lambda - \lambda; \phi_h, \theta_h)
\end{align*}
If we write $ (\lambda - \lambda_h, \psi) = m_{\psi,h}(\lambda - \pi_h \lambda)
+ m_{\psi,h}(\pi_h \lambda - \lambda_h)$,
we arrive at an error representation form similar
to~\eqref{eq:nitsche-flux-error-repres}:
\begin{equation*}
  (\lambda - \lambda_h, \psi)_{\partial \Omega}
  = A_h(\pi_h u - u, \pi_h \lambda - \lambda; \phi_h, \theta_h)
  - m_{\psi, h}(\lambda - \pi_h \lambda)
\end{equation*}
Consequently, the flux error
$ \|\sigma_n(u) - \Sigma_n(u_h) \|_{\partial \Omega}
=
\| \lambda - \lambda_h \|_{\partial \Omega}
$
can be estimated via following 
\begin{lemma}
    \label{lem:errorrep-lm}
  It holds
  \begin{equation}
    \label{eq:errorrep-lm}
    \| \lambda - \lambda_h \|_{\partial \Omega}
    \leqslant \sup_{\psi \in L^2(\partial\Omega) \setminus \{0\}} \frac{1}{\|\psi\|_{\partial \Omega}}
    \Big(
    |A_h(\pi_h u - u, \pi_h \lambda - \lambda; \phi_h, \theta_h)|
    + |m_{\psi,h}(\lambda - \pi_h \lambda)|
    \Big)
  \end{equation}
\end{lemma}

%--------------------------------------------------------------------------------------------------------
\section{Stability Bounds for the Discrete Dual Problem}
\label{sec:stability-bounds}
From the error representation formula stated in
Lemma~\ref{lemma:errorrep} and Lemma~\ref{lem:errorrep-lm}, we note
that in order to prove estimates for the flux in the $L^2$-norm, we
need to consider stability bounds in terms of the $L^2$-norm of
$\psi$. 
\citet{Chabrowski1991} proved such estimates for the corresponding continuous
problem: find $\phi : \Omega \rightarrow \R$ such that
\begin{align}
  \label{eq:contduala}
  -\Delta \phi &= 0 \quad \text{in $\Omega$}
  \\ \label{eq:contdualb}
  u&= \psi \quad \text{on $\partial \Omega$}
\end{align}
with $\psi \in L^2(\partial \Omega)$.
To state the basic energy type estimate, we shall introduce
some notation that will also be needed in our forthcoming developments.
Let $\rho(x) = \text{dist}(x,\partial \Omega)$ be the minimal distance
between $x \in \Omega$ and $\partial \Omega$ and $p(x)\in \partial \Omega$
be the point closest to $x \in \Omega$. We note that
$p(x) = x + n(p(x)) \rho(x)$,
where $n(p(x))$ is the exterior unit normal to $\partial \Omega$ at
$p(x)$, and that there is a constant $\delta_0>0$, only dependent on
the curvature of the boundary, such that for each $x\in \Omega$ with
$\rho(x) \leqslant \delta_0$ there is a unique $p(x) \in \partial \Omega$.
Next, we define the sets
\begin{equation}
\Omega_\delta = \{ x \in \Omega : \rho(x) > \delta \}
\label{eq:omega_delta}
\end{equation}
where $0\leqslant \delta \leqslant \delta_0$, and we note that the closest point
mapping $p: \partial \Omega_\delta \rightarrow \partial \Omega$
is a bijection with inverse denoted by $p^{-1}_\delta$.
Referring to \cite[Lemma 14.16]{GilbardTrudinger2001}, we recall that
that $\rho \in C^k(\Omega_{\delta})$ for
$k \geqslant 2$ for $\delta_0$ chosen small enough.
If we define a
weighted norm $\|v\|_{\rho,\Omega}$ by
\begin{equation*}
\|v \|^2_{\rho,\Omega} = \int_\Omega v^2 \rho \dx
\end{equation*}
then \citet{Chabrowski1991} proved the following result for the continuous problem:
if $\phi \in W^{1,2}_{\mathrm{loc}}$ satisfies
problem~\eqref{eq:contduala}--\eqref{eq:contdualb}
in the sense that
\begin{equation*}
\| \phi\circ p^{-1}_\delta - \psi\|_{\partial \Omega} \rightarrow 0
\quad \text{when $\delta \rightarrow 0^+$}
\end{equation*}
then
\begin{equation*}
\| \nabla \phi \|^2_{\rho,\Omega} + \| \phi \|_{\rho,\Omega}^2
+ \sup_{0 \leqslant \delta \leqslant \delta_0} \| \phi \|^2_{\partial \Omega_\delta}
\lesssim \| \psi \|^2_{\partial \Omega}.
\end{equation*}
We shall now prove a corresponding estimate for the discrete dual
problems~\eqref{eq:discrete-dual-problem-nitsche} and
\eqref{eq:discrete-dual-problem-lagrange}.
In order to formulate our results, we introduce the shifted weight
function
\begin{equation}
  \label{eq:shiftedweight}
  \rho_\delta = \max(0, \rho - \delta),\quad \deltahone \leqslant \delta \leqslant \delta_0
\end{equation}
and we let
\begin{equation}\label{eq:deltahone}
  \deltahone = C' h
\end{equation}
with the constant $C'>0$ chosen such that $\rho_{\deltahone}=0$ on all elements
with a face on the boundary $\partial \Omega$, see
Figure~\ref{fig:regions}.
The existence of such a constant $C'$
follows from the assumed quasi-uniformity of the mesh.
In the case where $\Omega$ is not a $C^2$-domain 
but rather a convex polyhedral domain described by faces $\{
F_i\}_{i=1}^N$, we define stripes 
$
S_{\delta}(F_i) = \{ x\in\R: \exists x_0 \in F_i \wedge \exists t \text{ s.t. } x = x_0
  + t\cdot n  \wedge 0\leqslant t \leqslant \delta \}
  $, cf. Figure~\ref{fig:poyhedron-with-stripes}.
Then the analysis presenting in this work carries over
by considering each stripe at a time and the fact that locally only a finite number of
stripes overlaps. 
\begin{figure}[htb]
    \centering
    \def\svgwidth{0.50\textwidth}
    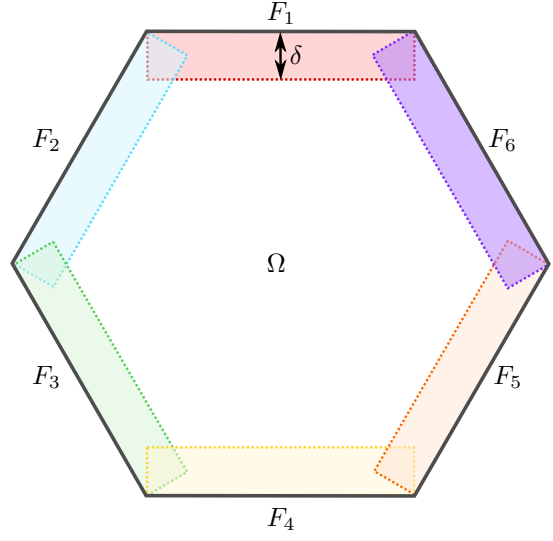
    \caption{Polyhedral domain $\Omega$ with faces
      $\{F_i\}_{i=1}^6$ with corresponding stripes
      $S_{\delta}(F_i)$ charted in different colors.
    }
    \label{fig:poyhedron-with-stripes}
\end{figure}
We state now the main result of this section.
\begin{proposition}
  \label{prop:stab}
  Let $\phi_h \in V_h$ be the solution
  of the discrete dual
  problem~\eqref{eq:discrete-dual-problem-nitsche}.
  Then $\phi_h$ satisfies the stability estimate
  \begin{equation}
    \label{eq:stabfinal-nitsche}
    \| \nabla \phi_h \|^2_{\rho_{\deltahone},\Omega}
    + h \| \nabla \phi_h \|^2_\Omega
    + \sup_{0 \leqslant \delta \leqslant \delta_0} \| \phi_h \|^2_{\partial \Omega_\delta}
    + \| \phi_h \|_{\Omega}^2
    \lesssim \| \psi \|^2_{\partial \Omega}
  \end{equation}
  Alternatively, if  $(\phi_h, \theta_h) \in V_h \times \Lambda_h$
  is the solution of the discrete dual
  problem~\eqref{eq:discrete-dual-problem-lagrange},
  then
  \begin{equation}
    \label{eq:stabfinal-lagrange}
    \| \nabla \phi_h \|^2_{\rho_\deltahone,\Omega}
    + h \| \nabla \phi_h \|^2_\Omega
    + h^2 \| \theta_h \|^2_{\partial \Omega}
    + \sup_{0 \leqslant \delta \leqslant \delta_0} \| \phi_h \|^2_{\partial \Omega_\delta}
    + \| \phi_h \|_{\Omega}^2
    \lesssim \| \psi \|^2_{\partial \Omega}
  \end{equation}
\end{proposition}
Before we present the elaborated proof of Proposition~\ref{prop:stab}
in Section~\ref{ssec:weighted-energy-estimate}, the next section
collects useful inequalities and interpolation estimates which will be
used throughout the remaining work.

\subsection{Interpolation Error Estimates}
We recall the following trace inequality for $v \in H^1(\Omega)$:
\begin{alignat}{1}
  \| v \|_{\partial T} &\lesssim h_{T}^{-1/2} \| v \|_{T}
  + h_{T}^{1/2} \| \nabla v \|_{T}  \quad \foralls T \in \mcT
  \label{eq:trace-inequality}
  \\
  \| v \|_{T \cap \partial \Omega} &\lesssim  h_{T}^{-1/2} \| v \|_{T}
  + h_{T}^{1/2} \| \nabla v \|_{T}
  \quad \foralls T \in \mcT
  \label{eq:trace-inequality-for-FD}
\end{alignat}
See~\citet{HansboHansbo2002} for a proof
of~\eqref{eq:trace-inequality-for-FD}.  We will also need the
following well-known inverse estimates for $v_h \in V_h$:
\begin{alignat}{3}
  \| \nabla v_h \|_T &\lesssim h_T^{-1} \| v_h \|_{T} &\quad
  &\foralls T \in \mcT
  \label{eq:inverse-estimates-for-triangles}
  \\
  \| h^{1/2} n \cdot \nabla v_h \|_{F} & \lesssim \| \nabla v_h
  \|_{T} &\quad
  &\foralls T \in \mcT
  \label{eq:inverse-estimate-for-facets}
\end{alignat}
Let $\pi_h:L^2(\Omega) \rightarrow V_h$ be the standard Scott--Zhang
interpolation operator~\citep{ScottZhang1990} and recall the
interpolation error estimates
\begin{alignat}{3}
  \| v - \pi_h v \|_{r,T} &\lesssim h^{s-r}| v |_{s,\omega(T)}
&\quad 0\leqslant r \leqslant s \leqslant 2 \quad &\foralls T\in \mcT
  \label{eq:interpest0}
  \\
\| v - \pi_h v \|_{r,F} &\lesssim h^{s-r-1/2}| v |_{s,\omega(T)}
&\quad 0\leqslant r \leqslant s \leqslant 2 \quad &\foralls F \in
  \mcF
  \label{eq:interpest1}
\end{alignat}
where $\omega(T)$ is the patch of neighbors of element $T$; that is,
the domain consisting of all elements sharing a vertex with $T$.

Recalling the definition~\eqref{eq:omega_delta} of $\Omega_{\delta}$,
we introduce the $h$-band
$\mcT_{\partial \Omega_{\delta}}$
for a mesh $\mcT$ by
\begin{align}
  \mcT_{\partial \Omega_{\delta}}
  = \bigcup\{T \in \mcT: T \cap \partial
  \Omega_{\delta} \neq \emptyset \}
  \label{eq:h-band}
\end{align}
This is illustrated in Figure~\ref{fig:regions}.
We note that thanks to the quasi-uniformity
\begin{equation*}
  |\mcT_{\partial \Omega_{\delta}} |_{d} \approx  h |
  \partial \Omega_{\delta} |_{d-1}
\end{equation*}
with $| \cdot |_{d}$ and $| \cdot |_{d-1}$ denoting the volume and
area of the corresponding sets.
The trace inequality~\eqref{eq:trace-inequality-for-FD} allows to
generalize the interpolation estimate~\eqref{eq:interpest1} to
\begin{equation}
  \| v - \pi_h v \|_{r,T\cap\partial \Omega_{\delta}} 
  \lesssim h^{s-r-1/2}| v |_{s,\omega(T)}
  \quad 0\leqslant r \leqslant s \leqslant 2 \quad \foralls T \in
  \mcT
  \label{eq:interpest1-unfitted}
\end{equation}
If we in addition assume that
\begin{align}
  \label{eq:anisotropic-regularity}
  \sup_{0 \leqslant \delta \leqslant \delta_1} \|
  D^s u \|_{\partial \Omega_{\delta} \cap \omega(T)} \lesssim 1
\end{align}
for some $\delta_1$ such that
$
\omega(T) \subset \bigcup_{0\leqslant\delta\leqslant\delta_1}
{\partial \Omega_{\delta}}
$, an order $h^{1/2}$ can be recovered in estimate~\eqref{eq:interpest1}
and ~\eqref{eq:interpest1-unfitted}
by applying H\"older's inequality in normal direction to
$\partial \Omega_{\delta}$:
\begin{equation}
  \| v - \pi_h v \|_{r,T\cap\partial \Omega_{\delta}}
  \lesssim h^{s-r} \sup_{0 \leqslant \delta \leqslant \delta_0} \|
  D^s u \|_{\partial \Omega_{\delta} \cap \omega(T)}
\quad 0\leqslant r \leqslant s \leqslant 2 \quad \foralls T \in
  \mcT
  \label{eq:interpest-anis-element}
\end{equation}
We summarize our observations in the following global,
anisotropic interpolation estimate:
\begin{proposition} 
  \label{prop:interpest-anis}
  Let $u \in H^2(\Omega)$ and suppose that
    $
    \sup_{0 \leqslant \delta \leqslant \delta_1}
    \| \partial^2 u \|_{\partial \Omega{\delta}} \lesssim 1
    $
    for some $\delta_1$ such that
    $\bigcup_{0\leqslant\delta \leqslant \delta_0 }
    \mcT_{\partial \Omega_{\delta}}
    \subset
    \bigcup_{0\leqslant\delta\leqslant\delta_1 } \partial \Omega_{\delta}.
    $
    Then the interpolation error satisfies
\begin{equation}
  \label{eq:interpest-anis}
    \sup_{0 \leqslant \delta \leqslant \delta_0}
    \|  u - \pi_h u \|_{\partial \Omega_{\delta}}
    +
    \sup_{0 \leqslant \delta \leqslant \delta_0}
    h \| \nabla( u - \pi_h u ) \|_{\partial \Omega_{\delta}}
  \lesssim
  h^2   
  \sup_{0 \leqslant \delta \leqslant \delta_1}
    \| \partial^2 u \|_{\partial \Omega_\delta}
\end{equation}
\end{proposition}
Note that the previous interpolation estimates holds if $u \in
W^{2,\infty}(\Omega)$ is the finite element solution
of~\eqref{eq:model-problem-weak-form} with \emph{strongly} imposed boundary
conditions, see \cite{RannacherScott1982}.
Here however, we only require that, roughly speaking,
$\partial^2 u \in L^2$ on manifolds close and parallel to the boundary
and $\partial^2 u \in L^{\infty}$ in normal direction
as quantified by assumption \eqref{eq:anisotropic-regularity}.
%\begin{proof}
%  For each $\delta$ and $r = 0, 1$,
%  the subsequent use of the trace
%  inequality~\eqref{eq:trace-inequality-for-FD}, the interpolation
%  estimates~\eqref{}--\eqref{} and H\"older's inequality in normal
%  direction of $\partial \Omega_{\delta}$ leads to
%  \begin{align*}
%    \| u - \pi_h u \|_{r,\partial \Omega_{\delta}}
%    &\lesssim
%    \sum_{T\in\mcT}\| u - \pi_h u \|_{r,\partial \Omega_{\delta} \cap T}
%    \\
%    &\lesssim
%    \sum_{T\in N_{\mcT}(\partial \Omega_{\delta})}
%    \left(
%    h^{-1/2}\| u - \pi_h u \|_{r,T}
%    +
%    h^{1/2}\|u - \pi_h u\|_{r+1,T}
%  \right)
%    \\
%    &\lesssim
%    \sum_{T\in N_{\mcT}(\partial \Omega_{\delta})}
%    h^{2-r-1/2} \| \partial^2 u \|_{\omega(T)}
%    \\
%    &\lesssim
%    \sum_{T\in N_{\mcT}(\partial \Omega_{\delta})}
%    h^{2-r}  \sup_{0 \leqslant \delta \leqslant \delta_1}
%    || \partial^2 u \|_{\partial \Omega_{\delta} \cap \omega(T)}
%  \\
%  &\lesssim
%  h^{2-r} \sup_{0 \leqslant \delta \leqslant \delta_1}
%  \| \partial^2 u \|_{\partial \Omega_{\delta}}
%  \end{align*}
%  Now taking the supremum over all $\delta \in (0,\delta_0)$
%  concludes the proof.
%\end{proof}

\subsection{Weighted Energy Stability}
\label{ssec:weighted-energy-estimate}
In this section, we finally prove Proposition~\ref{prop:stab}.
The main idea of the proof is to divide the domain into an interior
region and a boundary layer of thickness $O(h)$.
Away from the boundary, a
weighted stability estimate can be proven by
testing the discrete dual problems with a weighted test function.
This function is chosen such that it is identically zero in a layer of
elements next to the boundary and thus the boundary terms in the
discrete bilinear forms vanish. Since the desired weighted test
function does not reside in $V_h$ we approximate it with a Lagrange
interpolant and estimate the reminder.

Within the boundary layer, an estimate for the discrete energy
stability emanating from the coercivity of the finite element method
is established.  This stability scales with $h$ since the boundary
data only resides in $L^2$ but it holds all the way out to the
boundary. More specifically, the following lemma holds:
\begin{lemma}[Discrete Energy Stability]
  \label{lem:discrete-energy-stability}
  Let $\phi_h \in V_h$ be the solution of the discrete dual
  problem~\eqref{eq:duallambda}.
  Then for any $\kappa \geqslant 0$ it holds
\begin{equation}
  \label{eq:stabdiscrete}
  h \|\nabla \phi_h \|_\Omega^2
  +
  h^2 \|n \cdot \nabla \phi_h \|_{\partial \Omega}^2
  + h \kappa \| \phi_h \|_\Omega^2
  +\sup_{0 \leqslant \delta \leqslant \delta_h} \| \phi_h
  \|^2_{\partial \Omega_{\delta_h}}
  \lesssim \| \psi\|^2_{\partial \Omega}
\end{equation}
Alternatively, assume $(\phi_h, \theta_h) \in V_h \times Q_h$ is the
solution of the discrete dual problem~\eqref{eq:duallambdalagrange}.
Then for any $\kappa \geqslant 0$ it holds
\begin{equation}
  \label{eq:stabdiscretelagrange}
  h \|\nabla \phi_h \|_\Omega^2
  +
  h^2 \| \theta_h \|_{\partial \Omega}^2
  +
  h \kappa \| \phi_h \|_\Omega^2
  +\sup_{0 \leqslant \delta \leqslant \delta_h} \| \phi_h \|^2_{\partial \Omega_{\delta_h}}
  \leqslant
  C \| \psi\|^2_{\partial \Omega}
\end{equation}
\end{lemma}
\begin{proof}
We note that
the estimate
\begin{equation}\label{eq:stabaa}
h \tn \phi_h \tn^2 =
h \|\nabla \phi_h \|_\Omega^2 + h \kappa \| \phi_h \|_\Omega^2
+ h^2 \|n\cdot \nabla \phi_h \|^2_{\partial \Omega}
+ \|\phi_h\|_{\partial \Omega}^2  \lesssim \| \psi \|_{\partial \Omega}^2
\end{equation}
follows directly by setting $v=\phi_h$ in (\ref{eq:duallambda}),
using coercivity (\ref{eq:coercivity}),
and multiplying by $h$.
Furthermore, with
$0\leqslant \delta \leqslant \delta_0$ we find that
\begin{align}
  \label{eq:taylor-expansion}
  \| \phi_h \|^2_{\partial \Omega_\delta}
  \lesssim\| \phi_h \|^2_{\partial \Omega} + \delta \| \nabla \phi_h \|^2_{\Omega\setminus\Omega_\delta}
\end{align}
Thus for $0\leqslant \delta \leqslant \delta_h \lesssim h$, we have
\begin{equation}
  \label{eq:stabc}
  \sup_{0 \leqslant \delta \leqslant \delta_h} \| \phi_h \|^2_{\partial \Omega_{\delta_h}}
  \lesssim
  \| \phi_h \|^2_{\partial \Omega} + h \| \nabla \phi_h \|^2_{\Omega\setminus\Omega_{\delta_h}}
  \lesssim
  \| \psi\|^2_{\partial \Omega}
\end{equation}
Combining (\ref{eq:stabaa}) and (\ref{eq:stabc}) we arrive at the
desired estimate.

The second estimate~\eqref{eq:stabdiscretelagrange} can be shown
similarly.
Setting $(v,\mu) = (\phi_h,\theta_h)$ in
(\ref{eq:discrete-dual-problem-lagrange}),
using the inf-sup condition (\ref{eq:infsup}) and multiplying with
$h$, we directly obtain
\begin{equation}
h \tn (\phi_h, \theta_h)\tn^2 \lesssim \| \psi \|^2_{\partial \Omega}
\end{equation}
In particular, we have
\begin{equation}
h \|\nabla \phi_h \|_\Omega^2
+ h^2 \| \theta_h \|^2_{\partial \Omega}
+ h \kappa \| \phi_h \|_\Omega^2
+ \|\phi_h\|^2_{\partial \Omega} \lesssim \| \psi \|^2_{\partial \Omega}
\end{equation}
and using the estimate~\eqref{eq:taylor-expansion} once more, we arrive at the
desired estimate.
\end{proof}
\begin{proposition}
  \label{prop:stab-kappa}
  If $\phi_h \in V_h$ satisfies
  \begin{equation}
    \label{eq:duallambda}
    a_h(v,\phi_h) + \kappa(v,\phi_h) = m_{\psi,h}(v)\quad \forall v \in V_h
  \end{equation}
  or $(\phi_h, \theta_h) \in V_h \times \Lambda_h$ satifies
  \begin{equation}
    \label{eq:duallambdalagrange}
    A_h(v,\mu;\phi_h,\theta_h) + \kappa(v,\phi_h)
    = m_{\psi,h}(v,\mu)\quad \forall (v,\mu) \in V_h \times \Lambda_h
  \end{equation}
  with a constant large enough parameter $\kappa>0$. Then, in both cases,
  $\phi_h$ satisfies the stability estimate
  \begin{equation}
    \| \nabla \phi_h \|^2_{\rho_\deltahone,\Omega} + h \| \nabla \phi_h \|^2_\Omega
    + \sup_{0 \leqslant \delta \leqslant \delta_0} \| \phi_h \|^2_{\partial \Omega_\delta}
    + \| \phi_h \|_{\Omega}^2
    \lesssim \| \psi \|^2_{\partial \Omega}
  \end{equation}
\end{proposition}

\begin{proof}
  First, we note that discrete energy stability estimate
  provides sufficient control for $\delta_h \lesssim h$.
  Let now $\delta_h$ be chosen
  such that $0 < \deltahone < \delta_h$. Choosing the test function
  \begin{equation}
    v = I_h (\rho_\delta \phi_h) = \rho_\delta \phi_h +( I_h - I )\rho_\delta \phi_h, \quad \deltahone \leqslant \delta \leqslant \delta_0
  \end{equation}
  where $I_h$ is the Lagrange interpolant, in (\ref{eq:duallambda}) we obtain
  the identity
  \begin{align}
    0&=a_h(\phi_h,I_h(\rho_\delta \phi_h)) + \kappa (\phi_h,I_h(\rho_\delta \phi_h))_\Omega
    \\ \nonumber
    &= (\nabla \phi_h, \nabla I_h (\rho_\delta \phi_h))_\Omega
    + \kappa (\phi_h,I_h(\rho_\delta \phi_h))_\Omega
    \\ \nonumber
    &=
    \underbrace{(\nabla \phi_h, \nabla(I_h - I) (\rho_\delta \phi_h)) +
    \kappa (\phi_h,(I_h-I)(\rho_\delta \phi_h))_\Omega}_{I}
    \\ \nonumber
    &\qquad + \underbrace{(\nabla \phi_h, \nabla (\rho_\delta \phi_h) ) +
  \kappa (\phi_h,\rho_\delta \phi_h)_\Omega}_{II}
  \\
  &=I + II
\end{align}
Note that, due to our choice of $\deltahone$, $I_h  (\rho_\delta
\phi_h) = 0$ on all elements with a face on $\partial \Omega$ and thus
$m_{\psi}(\cdot)$ and the boundary terms in $a_h(\cdot,\cdot)$ and
vanish.
\paragraph{\bf Term $\bfI$} We divide the set of elements in the mesh $\mcT_h$
into three disjoint subsets
\begin{align*}
  \mcT_{0} &= \{T \in \mcT_h : \text{$\rho_{\delta}=0$ on $T$} \}
  \\
  \mcT_{\Omega_\delta} &= \{T \in \mcT_h : T \subset \text{supp}(\rho_{\delta})\}
  \\
  \mcT_{\partial \Omega_\delta} &= \mcT_h \setminus (\mcT_0 \cup \mcT_{\Omega_\delta})
\end{align*}
For each element, term $I$ can be estimated in the following way:
%\paragraph{\bf $T \in \mcT_{0}$:} 
\\
\noindent
$T \in \mcT_{0}$:\;
Clearly $(\nabla \phi_h,\nabla  (I_h - I)(\rho_{\delta} \phi_h ))_K=0$.
%\paragraph{\bf$T \in \mcT_{\Omega_\delta}$}
\\
\noindent
$T \in \mcT_{\Omega_\delta}$:\;
Using a standard interpolation
error estimate for the Lagrange interpolant, we conclude that
\begin{align}
  |(\nabla \phi_h,\nabla  (I_h - I)(\rho_{\delta} \phi_h ))_T|
  &\lesssim h \|\nabla \phi_h \|_{T}
  | \rho_{\delta} \phi_h |_{2,T}
  \nonumber
  \\
  &\lesssim h \|\nabla \phi_h \|_{T}
  \| \rho_{\delta} \|_{W^{2,\infty}(T)} \| \phi_h \|_{H^1(T)}
  \nonumber
  \\ \label{eq:termIa}
  &\lesssim h \Big( \|\nabla \phi_h \|^2_{T} + \|\phi_h \|_T^2 \Big)
  \quad \forall T \in \mcT_{\Omega_\delta}
\end{align}
$T \in \mcT_{\partial\Omega_\delta}$:\;
In this case
$\nabla \rho_\delta$ is discontinuous in $T$ and to deal with this
difficulty, we use Green's formula as follows
\begin{align}
  |(\nabla \phi_h,\nabla  (I_h - I)(\rho_\delta \phi_h ))_T|
  &= |(n\cdot \nabla \phi_h, (I_h - I)(\rho_\delta \phi_h ))_{\partial T}|
  \nonumber
  \\
  &\lesssim \| n\cdot \nabla \phi_h\|_{\partial T} \|(I_h - I)(\rho_\delta \phi_h )\|_{\partial T}
  \nonumber
  \\
  &\lesssim h^{-1/2} \| \nabla \phi_h\|_{T} \|(I_h - I)(\rho_\delta \phi_h )\|_{\partial T}
  \nonumber
  \\
  &\lesssim h^{1/2} \| \nabla \phi_h\|_{T}
  \|\nabla (\rho_\delta \phi_h) \|_{\partial T}
  \nonumber
  \\ \label{eq:termIb}
  &\lesssim \epsilon^{-1} h \| \nabla \phi_h\|^2_{T}  +
  \epsilon \|\nabla (\rho_\delta \phi_h) \|^2_{\partial T}
\end{align}
for each $\epsilon>0$. Here we used an inverse inequality and the interpolation estimate
\begin{equation*}
  \| v - I_v v \|_F \lesssim h \|\nabla_F v \|_F \lesssim \|\nabla v \|_F
\end{equation*}
on each of the faces $F\subset\partial T$ of element $T$. Here, $\nabla_F$
is the tangent gradient $\nabla_F v = P_F \nabla$ associated with the
face $F$ and $P_F = I - n_F \otimes n_F$, where $n_F$ is the unit normal to
$F$, the projection onto the tangent space of $F$.

Now $\|\nabla (\rho_\delta \phi_h) \|_{\partial T}$ can be estimated
by observing that
$\| \rho_\delta \|_{L^\infty(\partial T)}\lesssim h$ since
$T \in \mcT_{\partial \Omega_\delta}$.
Using H\"older's inequality, we
have
\begin{align}
  \|\nabla (\rho_\delta \phi_h) \|^2_{\partial T}
  &\lesssim
  \|\nabla \rho_\delta\|^2_{L^\infty(\partial T)} \| \phi_h \|^2_{\partial T}
  +  \| \rho_\delta \|^2_{L^\infty(\partial T)} \|\nabla \phi_h \|^2_{\partial T}
  \nonumber
  \\
  &\lesssim \|\phi_h \|^2_{\partial T} +  h \|\nabla \phi_h \|^2_{\partial T}
  \nonumber
  \\
  &\lesssim \Big( h^{-1}\|\phi_h \|^2_{T} + h \|\nabla \phi_h\|_T^2\Big)
  +  h \|\nabla \phi_h \|^2_{T}
  \nonumber
  \\ \label{eq:termIc}
  &\lesssim h^{-1}\|\phi_h \|^2_{T} +  h \|\nabla \phi_h\|_T^2
\end{align}
where we again used a trace inequality and an inverse estimate.

Combining
(\ref{eq:termIb}) and (\ref{eq:termIc}),  we thus have
\begin{equation}
  \label{eq:termId}
  |(\nabla \phi_h,\nabla  (I_h - I)(\rho_\delta \phi_h ))_T|
  \lesssim \epsilon h^{-1}\|\phi_h \|^2_{T} +  \epsilon^{-1} h \|\nabla \phi_h\|_T^2\quad \forall
  T \in \mcT_{\partial \Omega_\delta}
\end{equation}
for all $0< \epsilon \leqslant 1$.
Summing over the elements and using (\ref{eq:termIa}) and
(\ref{eq:termId}), we obtain
\begin{align}
  |I| &\lesssim \sum_{T \in {\mcT_{\Omega_\delta}}}
  h \Big( \| \nabla \phi_h \|^2_{T} + \| \phi_h \|^2_T \Big)
  +  \sum_{T \in {\mcT_{\partial \Omega_\delta}}}
  \Big( \epsilon h^{-1} \| \phi_h \|^2_{T} + \epsilon^{-1} h \| \nabla \phi_h \|^2_{T} \Big)
  \nonumber
  \\
  &\lesssim\epsilon^{-1} \sum_{T \in {\mcT_{h}}}
  h \Big( \| \nabla \phi_h \|^2_{T} + \| \phi_h \|^2_T \Big)
  + \epsilon  \sum_{T \in {\mcT_{\partial \Omega_\delta}}} \sup_{0\leqslant d \leqslant \delta_0} {\| \phi_h \|^2_{T\cap {\partial \Omega_d}}}
  \nonumber
  \\ \label{eq:termIe}
  &\lesssim \epsilon^{-1} \|\psi\|^2_{\partial \Omega}
  +  \epsilon \sup_{\delta_h\leqslant d \leqslant \delta_0} \|\phi_h\|^2_{\partial \Omega_d}
\end{align}
for all $0< \epsilon \leqslant 1$.

\paragraph{\bf Term $\bfI\bfI$} An application of Green's formula gives
the following identity
\begin{align}
  II&=(\nabla \phi_h, \nabla(\rho_{\delta} \phi_h))_{\Omega_\delta}
  + \kappa (\rho_\delta \phi_h, \phi_h)_{\Omega_\delta}
  \nonumber
  \\
  &=(\rho_{\delta}\nabla \phi_h,\nabla \phi_h)_{\Omega_\delta}
  + ( \phi_h \nabla \phi_h, \nabla \rho_\delta)_{\Omega_\delta}
  + \kappa (\rho_\delta \phi_h,\phi_h)_{\Omega_\delta}
  \nonumber
  \\
  &=(\rho_\delta \nabla \phi_h,\nabla \phi_h)_{\Omega_\delta}
  + \kappa (\rho_\delta \phi_h,\phi_h)_{\Omega_\delta}
  - \frac{1}{2}( \phi_h^2,\Delta \rho_\delta)_{\Omega_\delta}
  + \frac{1}{2}(\phi_h^2, (n \cdot \nabla \rho_\delta))_{\partial \Omega_\delta}
  \nonumber
\end{align}
We thus obtain the estimate
\begin{align}
  \|\nabla \phi_h\|^2_{\rho_\delta,\Omega}
  + \kappa \|\rho_\delta \phi_h\|^2_{\Omega}
  &\lesssim \|\phi_h\|^2_{\Omega_\delta} \|\Delta \rho_\delta\|_{L^\infty(\Omega_\delta)}
  +  \|\phi_h\|^2_{\partial \Omega_\delta} \| \nabla \rho_{\delta}\|_{L^\infty(\partial \Omega_\delta)} + |I|
\label{eq:termIIb}
  \\
  &\lesssim \|\phi_h\|^2_{\Omega_\delta}
  +  \|\phi_h\|^2_{\partial \Omega_\delta}
 \nonumber
  \\
  &\qquad  +   \epsilon^{-1} \|\psi\|^2_{\partial \Omega}
  +  \epsilon \sup_{\delta_h\leqslant d \leqslant \delta_0} \|\phi_h\|^2_{\partial \Omega_d}
 \label{eq:termIIc}
  \\
 \nonumber
  &\lesssim \|\phi_h\|^2_{\Omega_\delta}
  +   \epsilon^{-1} \|\psi\|^2_{\partial \Omega}
  +  \epsilon \sup_{\delta_h\leqslant d \leqslant \delta_0} \|\phi_h\|^2_{\partial \Omega_d}
\end{align}
for $\deltahone \leqslant \delta \leqslant \delta_h$.
Here we used the estimate (\ref{eq:termIe}) for Term $I$ in (\ref{eq:termIIb}) and the estimate (\ref{eq:stabc}) to bound $\|\phi_h\|^2_{\partial \Omega_\delta}$
for $\deltahone \leqslant \delta \leqslant \delta_h$ in (\ref{eq:termIIc}).
Thus, letting $\delta \rightarrow \deltahone$ we obtain
\begin{equation}
  \label{eq:staba}
  \| \nabla \phi_h \|^2_{\rho_\deltahone,\Omega} +\kappa \| \phi_h \|^2_{\rho_\deltahone,\Omega}
  \lesssim
  \|\phi_h\|^2_{\Omega_{\deltahone}}
  +   \epsilon^{-1} \|\psi\|^2_{\partial \Omega}
  +  \epsilon \sup_{\delta_h\leqslant d \leqslant \delta_0} \|\phi_h\|^2_{\partial \Omega_d}
\end{equation}
Using the fact $|n \cdot \nabla \rho | \geqslant c > 0$ for $\delta_0$
small enough, we also obtain
the bound
\begin{align}
  \nonumber
  \sup_{\deltahone \leqslant \delta \leqslant \delta_0} \| \phi_h \|^2_{\partial \Omega_\delta}
  &\lesssim \|\nabla \phi_h \|^2_{\rho_\deltahone,\Omega}
  +  \kappa \| \phi_h \|^2_{\rho_\deltahone,\Omega}
  + \|\phi_h\|^2_{\Omega_{\deltahone}}
  + |I|
  \\ \label{eq:stabbb}
  &\lesssim
   \|\phi_h\|^2_{\Omega_{\deltahone}}
  +   \epsilon^{-1} \|\psi\|^2_{\partial \Omega}
  +  \epsilon \sup_{\delta_h\leqslant d \leqslant \delta_0} \|\phi_h\|^2_{\partial \Omega_d}
\end{align}
where we used (\ref{eq:termIe}) and (\ref{eq:staba}) in the second
inequality. Choosing an appropriate $\epsilon$ and combining
(\ref{eq:staba}) and (\ref{eq:stabbb}), we arrive at
\begin{align}
  \label{eq:stabcc}
  &\| \nabla \phi_h \|^2_{\rho_\deltahone,\Omega}
  + \kappa \|\phi_h \|^2_{\rho_\deltahone,\Omega}
  +\sup_{\deltahone \leqslant \delta \leqslant \delta_0} \| \phi_h \|^2_{\partial \Omega_\delta}
  \lesssim
   \|\phi_h\|^2_{\Omega_{\deltahone}}
  + \|\psi\|^2_{\partial \Omega}
\end{align}
To conclude the proof, we first note that
$
\|\phi_h\|^2_{\Omega}
$
can be estimated by
\begin{align}
  \|\phi_h\|^2_{\Omega} &= \|\phi_h\|^2_{\Omega\setminus \Omega_d} +  \|\phi_h\|^2_{\Omega_d}
  \\
  &\lesssim d \sup_{0\leqslant \delta \leqslant d} \| \phi_h \|^2_{\partial \Omega_\delta}
  + d^{-1}\|\phi_h\|^2_{\rho, \Omega_d}
\end{align}
Applying the same argument for the domain $\Omega_{\deltahone}$ and the
shifted distance function $ \rho_{\deltahone}$, we note that by
choosing $\deltahone < d \leqslant \delta_0$ small enough and $\kappa$ large
enough, the term $\|\phi_h\|^2_{\Omega_{\deltahone}}$ can be absorbed
in the left hand side of (\ref{eq:stabcc}).
Thus we finally have the estimate
\begin{equation}\label{eq:stabweighted}
  \| \nabla \phi_h \|^2_{\rho_\deltahone,\Omega}
  + \kappa \|\phi_h \|^2_{\rho_\deltahone,\Omega}
  +\sup_{\deltahone \leqslant \delta \leqslant \delta_0} \| \phi_h \|^2_{\partial \Omega_\delta}
  \lesssim
  \|\psi\|^2_{\partial \Omega}
\end{equation}
The proof now follows from combining \eqref{eq:stabdiscrete} and
\eqref{eq:stabweighted}.
\end{proof}

We are now in the position to finalize the proof of
of Proposition~\ref{prop:stab}:
\begin{proof}{(Proposition \ref{prop:stab})}
  We decompose the solution $\phi_h$ to (\ref{eq:discrete-dual-problem-nitsche}) into a sum
  $\phi_h = \phi_{h,0} + \phi_{h,1}$ where
  \begin{equation}\label{eq:phih1}
    a_h(\phi_{h,1},v) + \kappa (\phi_{h,1},v) = m_{\psi,h}(v) \quad \forall v \in V_h
  \end{equation}
  and
  \begin{equation}\label{eq:phih0}
    a_h(\phi_{h,0},v) = \kappa (\phi_{h,1},v) \quad \forall v \in V_h
  \end{equation}
  Setting $v = \phi_{h,0}$ in (\ref{eq:phih0}) we find that
  \begin{equation}
    \tn \phi_{h,0} \tn^2 \lesssim \kappa \|\phi_{h,1}\|_{\Omega} \|\phi_{h,0}\|_{\Omega}
    \lesssim \| \psi \|_{\partial \Omega} \tn \phi_{h,0} \tn
  \end{equation}
  where we used Cauchy-Schwarz, Poincar\'e, and Proposition
  \ref{prop:stab-kappa}.
  Thus
  \begin{equation}
    \|\nabla \phi_{h,0} \|_\Omega^2 + h^{-1}\|\phi_{h,0} \|_{\partial \Omega}^2
    \lesssim
     \| \psi \|^2_{\partial \Omega}
  \end{equation}
  Using this estimate, we also obtain
  \begin{equation}
    \sup_{0\leqslant \delta \leqslant \delta_0} \| \phi_{h,0} \|^2_{\partial \Omega_\delta}
    \lesssim \|\psi \|^2_{\partial \Omega}
  \end{equation}
  Collecting the estimates we conclude that the estimate
  \begin{equation}\label{eq:phih0est}
    \|\nabla \phi_{h,0} \|_\Omega^2 + \sup_{0\leqslant \delta \leqslant \delta_0} \| \phi_{h,0} \|^2_{\partial \Omega_\delta}
    \lesssim \|\psi \|^2_{\partial \Omega}
  \end{equation}
  holds. Observing that this estimate is stronger compared to the
  desired estimate and the triangle inequality, the
  estimate (\ref{eq:phih0est}) for $\phi_0$ and the estimate
  for $\phi_{h,1}$ given by Proposition
  \ref{prop:stab-kappa} we obtain the desired result.
\end{proof}
\begin{figure}[htb]
  \centering
    \def\svgwidth{0.48\textwidth}
    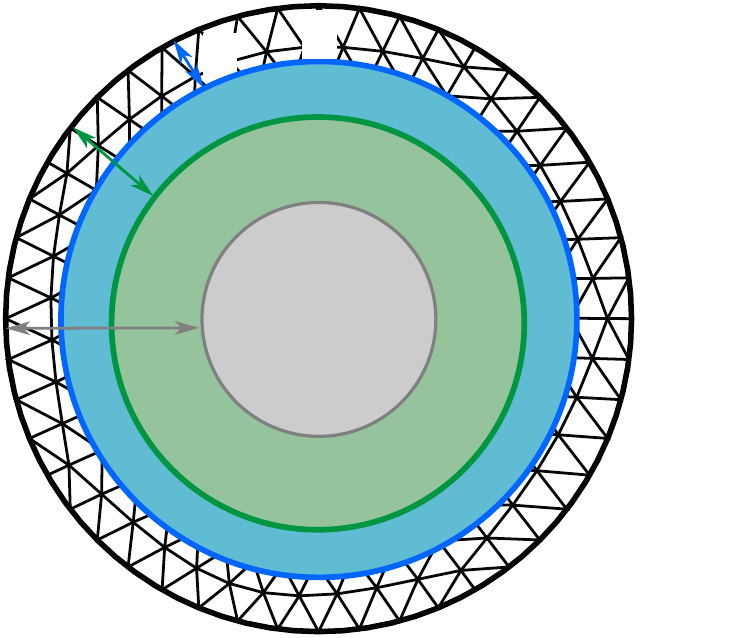
    \vspace{0.02\textwidth}
    \def\svgwidth{0.45\textwidth}
    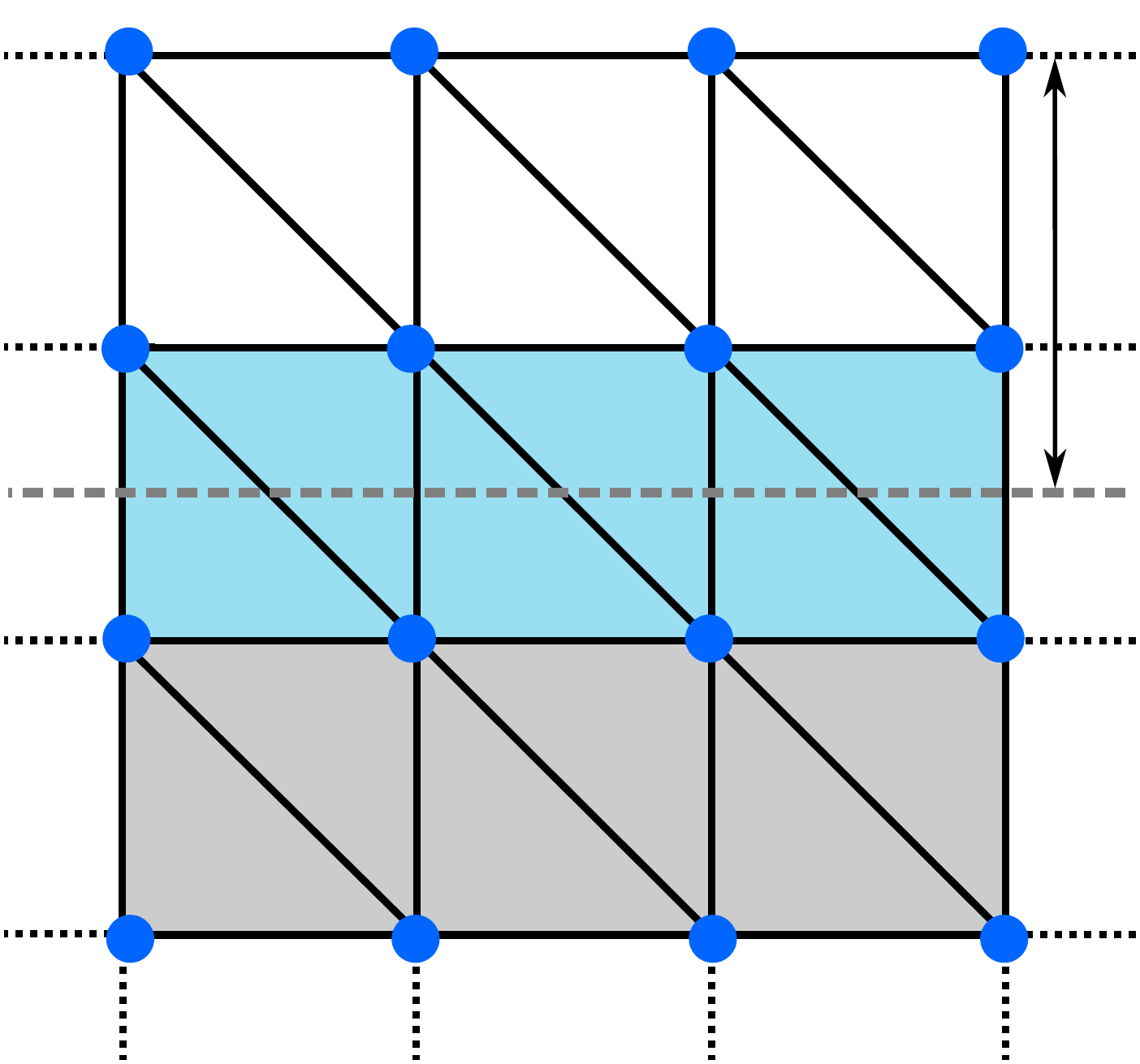
    \caption{(Left) Decomposition of the domain $\Omega$ into 
      two boundary layers of width $h$ and the ``far-field''
      domain $\Omega_{\delta_0}$.
      (Right) Decomposition of the mesh $\mcT$ with respect
      to $\partial \Omega_{\delta}$ consisting of
      a $h$-band $\mcT_{\partial \Omega_{\delta}}$ (blue), an inner
      mesh $\mcT_{\Omega_{\delta}}$ (gray) and a boundary zone
      $\mcT_0$ (white).}
    \label{fig:regions}
\end{figure}

\section{$L^{2}$ Error Estimates for the Boundary
Flux}
\label{sec:error-estimates}
The previous results on the weighted stability estimate and the anisotropic
interpolation error enable us to prove the main result of our work:
\begin{theorem} Let $\Sigma_n$ be the discrete boundary flux defined
  by either~\eqref{eq:flux-definition-nitsche-case} or
  ~\eqref{eq:flux-definition-lm-case} and suppose $u$ satisfies
  the assumption of Proposition~\ref{prop:interpest-anis}.
  Then the following error estimate holds
\begin{equation}
  \label{eq:flux-estimate}
\|\sigma_n - \Sigma_n\|_{\partial \Omega} \lesssim |\ln h | h
\end{equation}
\end{theorem}

\begin{proof}
  We start with the proof of the estimate for the Nitsche
  flux~\eqref{eq:flux-definition-nitsche-case}.
  Recalling the estimate (\ref{eq:errorrep}) in
  Lemma \ref{lemma:errorrep}, we need to estimate
  \begin{equation}
    I = \sup_{\psi \in L^2(\partial\Omega) \setminus \{0\}} \frac{|a_h(u - \pi_h u , \phi_h)|}{\|\psi\|_{\partial \Omega}},
    \quad
    II = \sup_{\psi \in L^2(\partial\Omega) \setminus \{0\}}\frac{|m_{\psi,h}(u - \pi_h u)| }{\|\psi\|_{\partial \Omega}}
  \end{equation}
  \paragraph{\bf Estimate of $\bfI$} We have
  \begin{align}
    a_h(u - \pi_h u, \phi_h) &=(\nabla (u - \pi_h u),\nabla \phi_h)_\Omega
    -(n\cdot \nabla ( u - \pi_h u ),\phi_h)_{\partial \Omega}
    \\ \nonumber
    &\qquad -(u - \pi_h u, n \cdot \nabla \phi_h )_{\partial \Omega}
    +\beta h^{-1}(u - \pi_h u, \phi_h )_{\partial \Omega}
  \end{align}
  The boundary terms may be directly estimated using a trace inequality
  followed by the interpolation error estimate (\ref{eq:interpest-anis}) and the
  stability estimate (\ref{eq:stabfinal-nitsche}). 
  For instance,
  \begin{align}
    |(n\cdot \nabla ( u - \pi_h u ),\phi_h)_{\partial \Omega}|
    \lesssim
    \| n\cdot \nabla ( u - \pi_h u )\|_{\partial \Omega} \|\phi_h \|_{\partial \Omega}
    \lesssim h \|\psi \|_{\partial \Omega}
    \\
    | (u - \pi_h u, n \cdot \nabla \phi_h )_{\partial \Omega} |
    \lesssim
    \| h^{-1} (u - \pi_h u) \|_{\partial \Omega} \| h n \cdot \nabla \phi_h
    \|_{\partial \Omega}
    \lesssim 
    h \|\psi \|_{\partial \Omega}
  \end{align}
  and the other terms may be estimated in the same way.
  To estimate the
  interior term we first split the integral as follows
  \begin{align}
    (\nabla (u - \pi_h u) ,\nabla \phi_h)_{\Omega}
    &=(\nabla (u - \pi_h u),\nabla \phi_h)_{\Omega \setminus \Omega_{\delta_h}}
    \\ \nonumber
    &\qquad +(\nabla (u - \pi_h u),\nabla \phi_h)_{\Omega_{\delta_h}\setminus \Omega_{\delta_0}}
    \\ \nonumber
    &\qquad
    + (\nabla (u - \pi_h u),\nabla \phi_h)_{\Omega_{\delta_0}}
    \\
    &= I_1 + I_2 + I_3
  \end{align}
  \paragraph{\bf Term $\bfI_1$}
  Using Cauchy-Schwarz in the tangent direction and
  H\"olders inequality in the normal direction we have
  \begin{align}
    I&=(\nabla (u - \pi_h u) ,\nabla \phi_h)_{\Omega \setminus \Omega_{\delta_h}}
    \\
    &\lesssim \left( \sup_{0\leqslant s \leqslant \delta_h} \| \nabla(u-\pi_h u)\|_{\partial\Omega_s} \right)
    \int_0^{\delta_h} \|\nabla \phi_h\|_{\partial \Omega_s} \ds
    \\
    &\lesssim \left( \sup_{0\leqslant s \leqslant \delta_h} \| \nabla(u-\pi_h u)\|_{\partial\Omega_s} \right)
    \left(\int_0^{\delta_h} \ds \right)^{1/2}
    \left(\int_0^{\delta_h} \|\nabla \phi_h\|_{\partial \Omega_s}^2 \ds\right)^{1/2}
    \\
    &\lesssim \left( \sup_{0\leqslant s \leqslant \delta_h} \| \nabla (u-\pi_h u)\|_{\partial\Omega_s}\right) \delta_h^{1/2}
    \|\nabla \phi_h\|_{\Omega \setminus \Omega_{\delta_h}}
  \end{align}
  Since $\delta_h = C h$ we may employ Proposition \ref{prop:stab} as follows
  \begin{equation}
    \delta_h
    \|\nabla \phi_h\|^2_{\Omega \setminus \Omega_{\delta_h}}
    \lesssim h \|\nabla \phi_h\|^2_{\Omega \setminus \Omega_{\delta_h}}
    \lesssim h \|\nabla \phi_h\|^2_{\Omega}
    \lesssim \|\psi \|^2_{\partial \Omega}
  \end{equation}
  Using the interpolation error estimate (\ref{eq:interpest-anis}) we
  get the estimate
  \begin{equation}
    |I|\lesssim h  \|\psi \|_{\partial \Omega}
  \end{equation}

  \paragraph{\bf Term $\bfI_2$} Proceeding in the same way as for
  Term $I$ and observing that
  \begin{equation}
    \rho_\deltahone(x) = s - \deltahone, \quad x \in \partial \Omega_s
  \end{equation}
  we get
  \begin{align}
    II&=(\nabla (u - \pi_h u) ,\nabla \phi_h)_{\Omega_{\delta_h}\setminus \Omega_\delta}
    \\
    &\lesssim \left( \sup_{\delta_h \leqslant s \leqslant \delta_0} \|\nabla( u-\pi_h u)\|_{\partial\Omega_s} \right)
    \int_{\delta_h}^{\delta_0} \|\nabla \phi_h\|_{\partial \Omega_s} \ds
    \\
    &\lesssim \left( \sup_{\delta_h \leqslant s \leqslant \delta_0}
  \|\nabla (u-\pi_h u) \|_{\partial\Omega_s} \right)
  \\ \nonumber
  &\qquad \times \left(\int_{\delta_h}^{\delta_0} (s-\deltahone)^{-1}\ds \right)^{1/2}
  \left(\int_{\delta_h}^{\delta_0} (s-\deltahone) \|\nabla \phi_h\|_{\partial \Omega_s}^2 \ds\right)^{1/2}
  \\
  &\lesssim h  |\ln \delta_h|^{1/2}
  \|\nabla \phi_h\|_{\rho_{\deltahone},\Omega \setminus \Omega_{\delta_h}}
  \\
  &\lesssim h  |\ln \delta_h|^{1/2} \| \psi \|_{\partial \Omega}
\end{align}
where we used the interpolation error estimate (\ref{eq:interpest-anis}) and
the stability estimate in Proposition \ref{prop:stab}.

\paragraph{\bf Term $\bfI_3$}
Using Cauchy-Schwarz we obtain
\begin{align}
  I&=(\nabla (u - \pi_h u) ,\nabla \phi_h)_{\Omega \setminus \Omega_\delta}
  \\
  &\lesssim \| \nabla (u-\pi_h u )\|_{\Omega \setminus \Omega_\delta} \delta^{-1/2}
  \|\nabla \phi_h\|_{\rho,\Omega \setminus \Omega_\delta}
\end{align}
which can be directly estimated using standard interpolation error estimates
and the stability bound.

\paragraph{\bf Estimate of $\bfI\bfI$} Using Cauchy-Schwarz and the
interpolation estimate (\ref{eq:interpest-anis}) we obtain
\begin{equation}
  |II|\lesssim \beta h^{-1} \|u - \pi_h u\|_{\partial \Omega} \| \psi
  \|_{\partial \Omega} \lesssim h \| \psi \|_{\partial \Omega}
\end{equation}
which concludes the proof.
\end{proof}

Following the same line of reasoning, we now state and prove the
corresponding $L^2$-error estimate when the boundary flux is
approximated by the Lagrange multiplier,
cf.~\eqref{eq:flux-definition-lm-case}.
Referring to the variational problem~\eqref{eq:lm-weak-form},
the stabilization form is supposed
the following localized version of the continuity
condition~\eqref{eq:continuity-lm}
\begin{align}
  | c_h(u,\lambda;v,\mu) |
  &\lesssim 
  \left(
    \| \nabla u \|_{\Omega_{\deltahone}}
    + \| h^{1/2} n \cdot \nabla u \|_{\partial \Omega}
    + \| h^{-1/2}  u \|_{\partial \Omega}
    + \| h^{1/2} \lambda \|_{\partial \Omega}
  \right)
  \nonumber
  \\
  &\quad \cdot
  \left(
    \| \nabla v \|_{\Omega_{\deltahone}}
    + \| h^{1/2} n \cdot \nabla v \|_{\partial \Omega}
    + \| h^{-1/2} v \|_{\partial \Omega}
    + \| h^{1/2} \mu \|_{\partial \Omega}
  \right)
  \label{eq:stab-form-local-cont}
\end{align}
This assumptions is trivially satisfies by the stabilization
form~\eqref{eq:stabilization-form-stenberg} and merely quantifies that
the region of influence of the stabilization is located
on or close to the boundary.

\begin{theorem} Let $\Sigma_n$ be the discrete boundary flux
  defined~\eqref{eq:flux-definition-lm-case} and assume the $u$ satisfies
  the assumption of Proposition~\ref{prop:interpest-anis} and 
  that $\lambda \in H^1(\pO)$.
  Then the following error estimate holds
\begin{equation}
  \label{eq:flux-estimate-lm}
\|\sigma_n - \Sigma_n\|_{\partial \Omega} \lesssim |\ln h | h
\end{equation}
\end{theorem}

\begin{proof}
Starting from the error representation
formula~\eqref{eq:errorrep-lm}, we need to estimate
\begin{equation*}
  I = \sup_{\psi \in L^2(\partial \Omega)\setminus 0}
  \dfrac{|A_h(\pi_h u -u, \pi_h \lambda - \lambda; \phi_h, \theta_h)
  |}
  {\| \psi \|_{\partial \Omega}},
  \quad
  II = \sup_{\psi \in L^2(\partial \Omega)\setminus 0}
  \dfrac{m_{\psi,h}(\lambda - \pi_h \lambda) }
  {\| \psi \|_{\partial \Omega}},
\end{equation*}
By definition, 
\begin{align*}
A_h(\pi_h u -u, \pi_h \lambda - \lambda; \phi_h, \theta_h)
&= (\nabla (\pi_h u - u), v)_{\Omega} 
+ (\pi_h \lambda - \lambda; \phi_h)_{\partial \Omega}
+ (\theta_h,  \pi_h u - u)_{\partial \Omega}
\\
&\quad - c_h(\pi_h \lambda - \lambda, \pi_h u - u ; \phi_h, \theta_h)
\end{align*}
Since the estimate for first term has already been derived in the previous
proof, it remains to bound the contribution from the
boundary terms and the stabilization form. 
An application of
the interpolation estimates and the discrete energy
stability~\eqref{eq:stabdiscretelagrange} yields
\begin{align*}
(\pi_h \lambda - \lambda; \phi_h)_{\partial \Omega}
&\lesssim h 
\| \lambda \|_{1, \partial \Omega} \| \psi \|_{\partial \Omega}
\\
(\theta_h, \pi_h u - u)_{\partial \Omega}
&\lesssim
\| h \theta \|_{\pO} \| h^{-1} \pi_h u - u \|_{\pO}
\lesssim  h \| \psi \|_{\pO} 
\end{align*}
Because of assumption~\eqref{eq:stab-form-local-cont},
the contribution from the stabilization form can be estimated similarly.
Finally, thanks to an interpolation estimate, term $II$ trivially satisfies
$| II | \lesssim h \|\lambda\|_{\pO}$.
\end{proof}

\section{Numerical Results}
\label{sec:numerical-results}
We consider the elliptic model
problem~\eqref{eq:model-problem-strong-pde}--\eqref{eq:model-problem-strong-bc}
on the domain $\Omega=[0,1]\times[0,1] \subset \R^2$.
To examine the convergence rate of the normal flux approximations,
we employ the method of manufactured solution and choose
\begin{align*}
  u(x,y) = \cos(2\pi x)\cos(2\pi y) + \sin(2 \pi x) \sin(2\pi y)
\end{align*}
as a reference solution, $g = u|_{\partial \Omega}$ and $f = -\Delta
u$ as the corresponding boundary data and source function,
respectively.

As discretization schemes, we pick Nitsche's
method~\eqref{eq:nitsche-form} and a stabilized Lagrange multiplier
method~\eqref{eq:lm-weak-form} with the stabilization form given by
\eqref{eq:stabilization-form-stenberg}.
For the stabilization parameters 
we take $\alpha = \beta = 10$. 
The approximations for the boundary flux are then computed on a sequences of uniform meshes
$\{\mcT_h\}_h$ with mesh sizes $ h \approx \tfrac{1}{4\sqrt{2}^{k}}$ for
$k=0,\dots,15$.
The numerical results are depicted in 
Figure~\ref{fig:fluxes-l2-convergence}.
In the pre-asymptotic regime ranging from $h\approx 0.35$ to $h\approx
0.1$, the convergence rate of both methods deviates significantly from
the optimal slope $1.0$. Consequently, the fitted slopes
indicate a slightly sub-optimal convergence rate for the Nitsche flux,
while the convergence rate for Lagrange multiplier method is higher then the
theoretical prediction.
If we discard the pre-asymptotic regime
as shown in the right plot of Figure~\ref{fig:fluxes-l2-convergence},
the approximation error $\|
\sigma_n - \Sigma_h\|_{\partial \Omega}$ exhibits  optimal convergence
rate for both methods and corroborates the theoretical findings of our
work.
\begin{figure}[htb]
  \centering
  \includegraphics[width=0.60\textwidth]{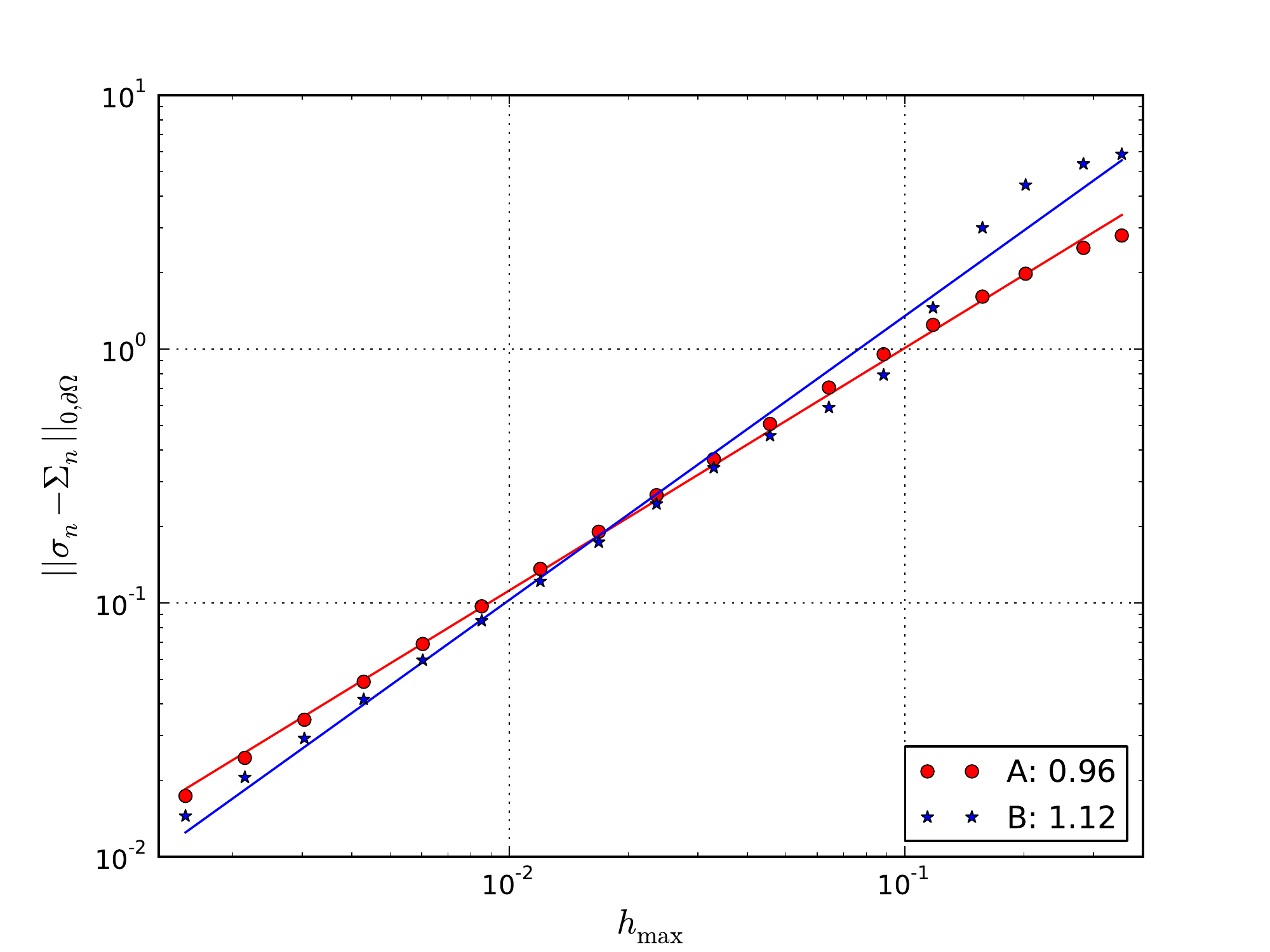}
  \includegraphics[width=0.60\textwidth]{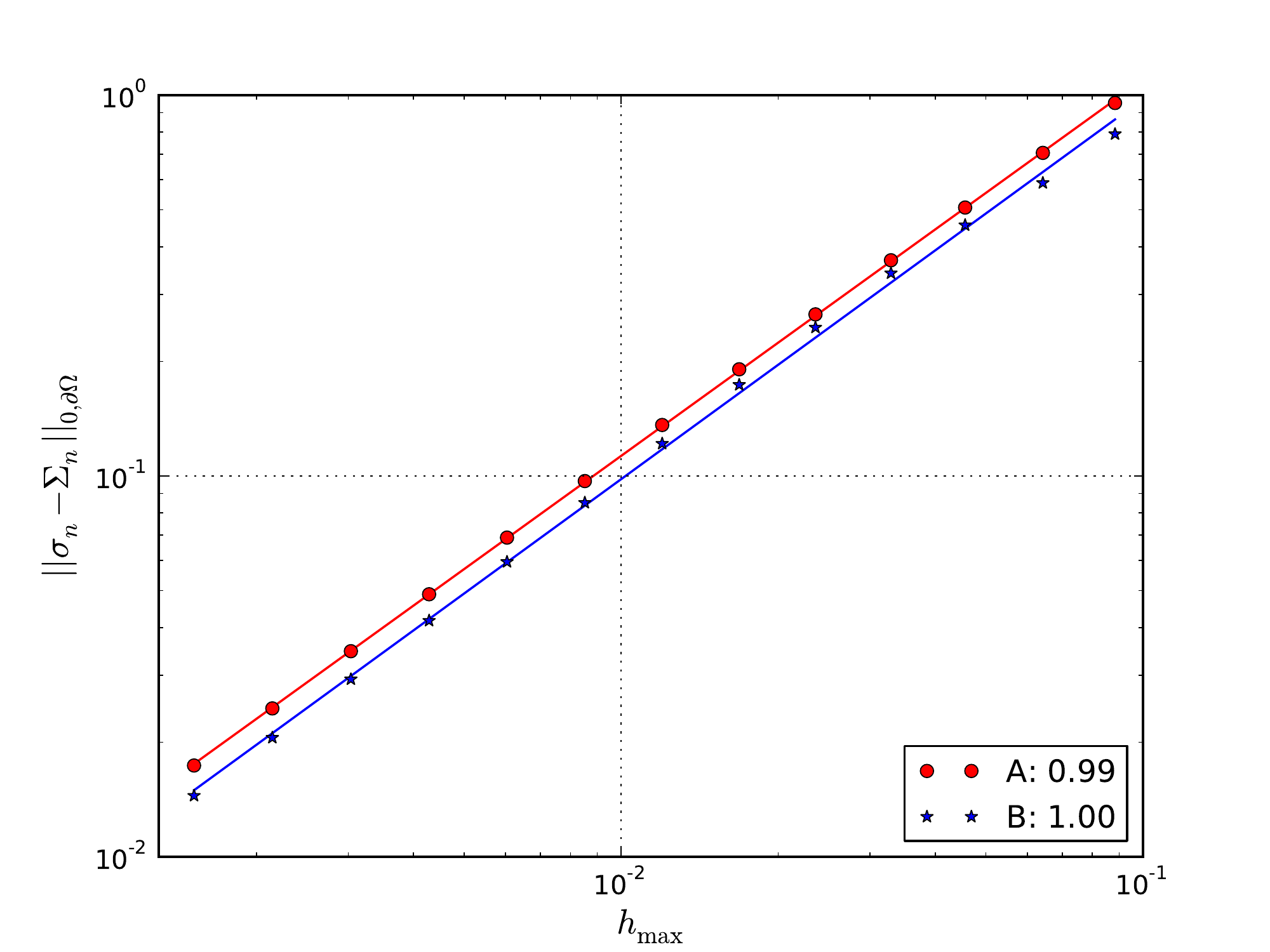}
  \caption{$L^2(\partial \Omega)$ convergence study for various flux
    computations.  (A) Nitsche flux for $\mathrm{CG}(1)$ elements. (B) Lagrange multiplier 
    computed with the stabilized method by Barbosa and Hughes based on
    a $\mathrm{CG}(1) \times \mathrm{DG}(0)$ discretization.
    The legend gives the fitted slope for each approximation error.
    (Left) Approximation error for the entire mesh sequence revealing
    different behavior in the pre-asymptotic regime.
    (Right) Asymptotic regime. Starting from $h\approx 0.1$
    both methods give optimal first order convergence.
  }
  \label{fig:fluxes-l2-convergence}
\end{figure}
\clearpage

\section*{Acknowledgments}
This work is supported by a Center of Excellence grant from the Research
Council of Norway to the Center for Biomedical Computing at Simula Research
Laboratory.

\bibliographystyle{plainnat}
\bibliography{bibliography}

\end{document}